\documentclass[12pt]{article}

\usepackage{latexsym,amssymb,upref,amsmath,amsthm, amsfonts,authblk}
\usepackage{amssymb,amsmath,amsthm, calc, graphicx}
\usepackage{epsfig}
\usepackage{breqn}
\usepackage{footnpag}
\usepackage{rotating}
\usepackage{amsfonts}
\usepackage{setspace}
\usepackage{fullpage}
\usepackage{enumitem}
\usepackage{bbold}
\usepackage{comment}
\usepackage{pgf,tikz}
\usepackage{mathrsfs}
\usetikzlibrary{arrows}
\usepackage{yhmath}
\usepackage{hyperref}
\usepackage{authblk}

\newtheorem{thm}{Theorem}

\newtheorem{claim}{Claim}

\newcommand{\thistheoremname}{}
\newtheorem*{genericthm*}{\thistheoremname}
\newenvironment{namedthm*}[1]
{\renewcommand{\thistheoremname}{#1}%
	\begin{genericthm*}}
	{\end{genericthm*}}

\title{A note on the maximum number of triangles in a $C_5$-free graph}
\author
{
Beka Ergemlidze
\thanks{ Department of Mathematics, Central European University, Budapest.
		E-mail: \texttt{beka.ergemlidze@gmail.com}} \qquad 
Ervin Gy\H{o}ri \thanks{R\'enyi Institute, Hungarian Academy of Sciences and 
	Department of Mathematics, Central European University, Budapest. E-mail: \texttt{gyori.ervin@renyi.mta.hu}} \qquad 
Abhishek Methuku \thanks{Department of Mathematics, Central European University, Budapest. (Corresponding) E-mail: \texttt{abhishekmethuku@gmail.com}} \qquad
Nika Salia \thanks{Department of Mathematics, Central European University, Budapest. E-mail: \texttt{Nika\char`_Salia@phd.ceu.edu}}
}

\begin{document}

\maketitle

\begin{abstract}
We prove that the maximum number of triangles in a $C_5$-free graph on $n$ vertices is at most $\frac{1}{2 \sqrt 2} (1 + o(1)) n^{3/2}$, improving an estimate of Alon and Shikhelman \cite{AlonS}.
\end{abstract}

\section{Introduction}
Motivated by a conjecture of Erd\H os  \cite{Erd} on the number of pentagons in triangle-free graphs, Bollob\'as and Gy\H ori \cite{BolGy} initiated the study of the converse of this problem. Let $ex(n,K_3,C_5)$ denote the maximum possible number of triangles in a graph on $n$ vertices without containing a cycle of length five as a subgraph. Bollob\'as and Gy\H ori \cite{BolGy} showed that 

\begin{equation}
\label{eq:BGY}
\frac{1}{3 \sqrt 3} (1 + o(1)) n^{3/2} \le ex(n,K_3,C_5) \le \frac{5}{4} (1 + o(1)) n^{3/2}.
\end{equation}

Their lower bound comes from the following example: Take a $C_4$-free bipartite graph $G_0$ on $n/3 + n/3$ vertices with about $(n/3)^{3/2}$ edges and double each vertex in one of the color classes and add an edge joining the old and the new copy to produce a graph $G$. Then, it is easy to check that $G$ contains no $C_5$ and the number of triangles in $G$ is the same as the number of edges in $G_0$.

Recently, Alon and Shikhelman \cite{AlonS} improved the above result by showing that
\begin{equation}
\label{eq:AS}
ex(n,K_3,C_5) \le \frac{\sqrt 3}{2} (1 + o(1)) n^{3/2}.
\end{equation}

In fact, in their nice paper, they investigate the more general function $ex(n,T,H)$ which stands for the maximum possible number of copies of $T$ in a $H$-free graph on $n$ vertices. 

\vspace{3mm}

In this note we improve \eqref{eq:BGY} and \eqref{eq:AS} by showing that,

\begin{thm}
	\label{Main_Result}
\begin{displaymath}
ex(n,K_3,C_5) \le \frac{1}{2 \sqrt 2} (1 + o(1)) n^{3/2}.
\end{displaymath}
\end{thm}

In an upcoming paper we prove more results following the approach introduced in this note and focus on improving our bound in Theorem \ref{Main_Result} further.

\vspace{3mm}

Our main idea is to select an appropriate subgraph such that the number of edges in the subgraph is the same as the number of triangles in the original graph and then we apply the following well-known theorem of Erd\H os and Simonovits \cite{Erd_Sim}. 

\begin{thm}(Erd\H os,  Simonovits \cite{Erd_Sim})
	\label{Erdos_Sim}
The maximum possible number of edges in a graph on $n$ vertices containing no $C_4$ or $C_5$ as a subgraph is at most $\frac{1}{2 \sqrt2} (1 + o(1)) n^{3/2}$.
\end{thm}

\section{Proof of Theorem \ref{Main_Result}}
Let $G$ be a $C_5$-free graph with maximum possible number of triangles. We may assume that each edge of $G$ is contained in a triangle, because otherwise, we can delete it without changing the number of triangles. Two triangles $T, T'$ are said to be in the same \emph{block} if they either share an edge or if there is a sequence of triangles $T, T_1, T_2, \ldots, T_s, T'$ where each triangle of this sequence shares an edge with the previous one (except the first one of course). It is easy to see that all the triangles in $G$ are partitioned uniquely into blocks. Notice that any two blocks of $G$ are edge-disjoint.  Below we will characterize the blocks of $G$.

A block of the form $\{abc_1, abc_2, \ldots, abc_k\}$ where $k \ge 1$, is called a \emph{crown-block} (i.e., a collection of triangles containing the same edge) and a block consisting of all triangles contained in the complete graph $K_4$ is called a \emph{$K_4$-block}. See Figure \ref{figure1}.

\begin{figure}[h]
	\begin{center}
		\includegraphics[scale=0.25]{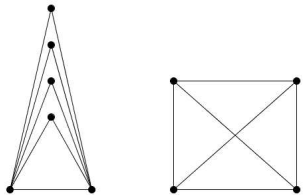}
	\end{center}
	\caption{An example of a crown-block and a $K_4$-block}
	\label{figure1}
\end{figure}

\begin{claim}
	\label{Blocks}
Every block of $G$ is either a crown-block or a $K_4$-block.
\end{claim}
\begin{proof}
If a block contains only one or two triangles, then it is easy to see that it is a crown-block. So we may assume that a block of $G$ contains at least three triangles and let $abc_1, abc_2$ be some two triangles in it. We claim that if $bc_1x$ or $ac_1x$ is a triangle in $G$ which is different from $abc_1$, then $x = c_2$. Indeed, if $x \not = c_2$, then the vertices $a,x,c_1,b,c_2$ contain a $C_5$, a contradiction. Similarly, if $bc_2x$ or $ac_2x$ is a triangle in $G$ which is different from $abc_2$, then $x = c_1$. 

Therefore, if $ac_i$ or $bc_i$ (for $i = 1, 2$) is contained in two triangles, then $abc_1c_2$ forms a $K_4$. However, then there is no triangle in $G$ which shares an edge with this $K_4$ and is not contained in it because if there is such a triangle, then it is easy to find a $C_5$ in $G$, a contradiction. So in this case, the block is a $K_4$-block, and we are done. 

So we can assume that whenever $abc_1, abc_2$ are two triangles then the edges $ac_1, bc_1, ac_2, bc_2$ are each contained in exactly one triangle. Therefore, any other triangle which shares an edge with either $abc_1$ or $abc_2$ must contain $ab$. Let $abc_3$ be such a triangle. Then applying the same argument as before for the triangles $abc_1, abc_3$ one can conclude that the edges $ac_3, bc_3$ are contained in exactly one triangle and so, any other triangle of $G$ which shares an edge with one of the triangles $abc_1, abc_2, abc_3$ must contain $ab$ again. So by induction, it is easy to see that all of the triangles in this block must contain $ab$. Therefore, it is a crown-block, as needed.
\end{proof}

Recall that any two blocks of $G$ are edge-disjoint. We claim the following.

\begin{claim}
	\label{C_4_inside_block}
The edges of any $C_4$ in $G$ are contained in only one block of $G$.
\end{claim}

\begin{proof}
Let $xyzw$ be a $4$-cycle in $G$. Every edge of $G$ is contained in a triangle. So in particular, let $xyu$ be a triangle containing the edge $xy$. If $u \not \in \{x,y,z,w\}$ then $uxwzy$ is a $C_5$, a contradiction. Therefore, $u = z$ or $u = w$. So either $xyz$ and $yzw$ or $xyw$ and $ywz$ are triangles of $G$. In both cases, the two triangles share an edge, so they belong to the same block. Hence, all four edges of $xyzw$ lie in the same block.
\end{proof}

We are now ready to prove the theorem using the above claims. We want to select a $C_4$-free subgraph $G_0$ of $G$ such that the number of edges in $G_0$ is the same as the number of triangles in $G$. By Claim \ref{Blocks} the edge set of every $C_4$ is completely contained in some block of $G$. So in order to make sure the selected subgraph $G_0$  is $C_4$-free, it suffices to make sure the edges selected from each block of $G$ do not contain a $C_4$, which is done as follows: From each crown-block $\{abc_1, abc_2,\ldots, abc_k\}$, we select the edges $ac_1, ac_2,\ldots,ac_k$ to be in $G_0$. From each $K_4$-block $abcd$ we select the edges $ab, bc, ac, ad$ to be in $G_0$ (since every block is either a crown-block or a $K_4$-block by Claim \ref{Blocks}, we have dealt with all the blocks of $G$). Finally, notice that the number of selected edges in each block is exactly the number of triangles in that block. Moreover, since blocks are edge-disjoint, we never select the same edge twice. Therefore, as every triangle of $G$ is contained in some block, the total number of triangles in $G$ is the same as the number of edges in $G_0$. On the other hand, as $G_0$ is $C_4$-free and also $C_5$-free (as it is a subgraph of $G$), we can use Theorem \ref{Erdos_Sim}, to obtain that the number of edges in it is at most $\frac{1}{2 \sqrt2} (1 + o(1)) n^{3/2}$, completing the proof of Theorem \ref{Main_Result}.


\section*{Acknowledgements}
The research of the second and third authors is partially supported by the National Research, Development and Innovation Office  NKFIH, grant K116769.

\end{document}